\documentclass[12pt]{article}
\usepackage{amssymb, amsmath, amsthm, amsfonts,xcolor, enumerate}
\usepackage{tikz}
\usetikzlibrary{patterns}
\usetikzlibrary{shapes}
\usepackage{fullpage}
\usepackage{hyperref}
\usepackage{enumitem}

\newtheorem{theorem}{Theorem}[section]
\newtheorem{lemma}[theorem]{Lemma}

\theoremstyle{definition}

\newtheorem{conj}[theorem]{Conjecture}
\newtheorem{claim}[theorem]{Claim}

\theoremstyle{definition}
\newtheorem{rmk}[theorem]{Remark}
\theoremstyle{definition}

\theoremstyle{definition}
\newtheorem{construction}[theorem]{Construction}

\theoremstyle{definition}
\newtheorem{definition}[theorem]{Definition}
\theoremstyle{definition}

\theoremstyle{definition}

\theoremstyle{definition}

\newcommand\BE{\ensuremath{\mathrm{BE}}}
\newcommand\ex{\ensuremath{\mathrm{ex}}}
\newcommand\RT{\ensuremath{\mathrm{RT}}}
\newcommand\RF{\ensuremath{\mathrm{RF}}}

\newcommand{\ep}{\varepsilon}
\newcommand{\om}{\omega}

\newcommand{\al}{\alpha}
\newcommand{\de}{\delta}

\newcommand{\ga}{\gamma}
\newcommand{\De}{\Delta}

\newcommand{\cH}{\mathcal{H}}

\newcommand{\cT}{\mathcal{T}}

\newcommand{\cA}{\mathcal{A}}
\newcommand{\cB}{\mathcal{B}}
\newcommand{\cP}{\mathcal{P}}

\newcommand{\cC}{\mathcal{C}}
\newcommand{\Ho}{\H{o}}

\newcommand{\bS}{\ensuremath{\mathbb{S}}}

\newcommand{\floor}[1]{\left\lfloor #1 \right\rfloor}
\newcommand{\ceiling}[1]{\left\lceil #1 \right\rceil}

\title{On two problems in Ramsey-Tur\'an theory}
 \author{
 J\'ozsef Balogh
 \thanks{Department of Mathematical Sciences,
 University of Illinois at Urbana-Champaign, Urbana, Illinois 61801, USA. Email: {\tt jobal@math.uiuc.edu}. Research is partially supported by NSF Grant DMS-1500121, Arnold O. Beckman Research Award (UIUC Campus Research Board 15006).
  } 
 \quad\quad
 Hong Liu
 \thanks{Mathematics Institute and DIMAP, University of Warwick, Coventry, CV4 7AL, UK.  Email: {\tt h.liu.9@warwick.ac.uk}. Supported partly by ERC grant 306493, EPSRC grant EP/K012045/1 and the Leverhulme Trust Early Career Fellowship~ECF-2016-523.}
 \quad\quad
 Maryam Sharifzadeh
 \thanks{Department of Mathematical Sciences,
 	University of Illinois at Urbana-Champaign, Urbana, Illinois 61801, USA. Email: {\tt sharifz2@illinois.edu}.}
 }

\begin{document}
\maketitle

\begin{abstract}
Alon, Balogh, Keevash and Sudakov proved that the $(k-1)$-partite Tur\'an graph maximizes the number of distinct $r$-edge-colorings with no monochromatic $K_k$ for all fixed $k$ and $r=2,3$, among all $n$-vertex graphs. In this paper, we determine this function asymptotically for $r=2$ among $n$-vertex graphs with sub-linear independence number. Somewhat surprisingly, unlike  Alon-Balogh-Keevash-Sudakov's result, the extremal construction from Ramsey-Tur\'an theory, as a natural candidate, does not maximize the number of distinct edge-colorings with no monochromatic cliques among all graphs with sub-linear independence number, even in the 2-colored case. 

In the second problem, we determine the maximum number of triangles asymptotically in an $n$-vertex $K_k$-free graph $G$ with $\alpha(G)=o(n)$. The extremal graphs have similar structure to the extremal graphs for the classical Ramsey-Tur\'an problem, i.e.~when the number of edges is maximized.
\end{abstract}

\section{Introduction}
Numerous classical problems in extremal graph theory have highly structured extremal configurations. For example, Tur\'an~\cite{Tu} in 1941 proved that $\ex(n,K_k)$, the maximum number of edges in an $n$-vertex $K_k$-free graph, is attained only by the balanced complete $(k-1)$-partite graph, known now as the \emph{Tur\'an graph} $T_{n,k-1}$. Motivated by the fact that the Tur\'an graph is particularly symmetric, admitting a $(k-1)$-partition into linear-sized independent sets, Erd\H os and S\'os~\cite{ES} introduced \emph{Ramsey-Tur\'an} type questions, where they investigated the maximum size of a $K_k$-free graph $G$ with the additional condition that $\alpha(G)=o(|G|)$.

Denote by $\RT(n,K_k, o(n))$ the Ramsey-Tur\'an function for $K_k$, i.e.~the maximum size of an $n$-vertex $K_k$-free graph with independence number $o(n)$. In 1970, Erd\Ho s and S\'os~\cite{ES} determined $\RT(n,K_k,o(n))$ for every odd $k$. The problem becomes much harder when an even clique is forbidden. For $k=4$, Szemer\'edi~\cite{Sz-k4}, using the regularity lemma, proved that $\RT(n,K_4,o(n))\le n^2/8+o(n^2)$. It had remained an open question whether $\RT(n,K_4,o(n))=\Omega(n^2)$. Bollob\'as and Erd\Ho s, in their seminal work~\cite{BE}, constructed a dense, $K_4$-free graph with sub-linear independence number, matching the upper bound above  (see Section~\ref{sec-pre} for more details). For all even $k$, the order of magnitude of $\RT(n,K_k,o(n))$ was finally determined by Erd\Ho s, Hajnal, S\'os and Szemer\'edi~\cite{EHSSz} in 1983. See~\cite{SS} for a survey and~\cite{Jozsi-1, Jozsi-2} for more recent developments on this topic.

In this paper, we will study Ramsey-Tur\'an extensions of some classical results, whose extremal graphs are close to the Tur\'an graph. See e.g.~\cite{BMSh} for one such extension of a graph tiling problem.

\subsection{Edge-colorings forbidding monochromatic cliques}
Denote by $F(n,r,k)$ the maximum number of $r$-edge-colorings that an $n$-vertex graph can have 
without a monochromatic copy of $K_k$. A trivial lower bound is given by $T_{n,k-1}$ as every $r$-edge-coloring of a $K_k$-free graph is monochromatic $K_k$-free: $F(n,r,k)\ge r^{\ex(n,K_k)}$. Erd\Ho s and Rothschild~\cite{E(R)} in 1974 conjectured that, for sufficiently large $n$, the above obvious lower bound is optimal for 2-edge-colorings. This was verified for $k=3$ by Yuster~\cite{Y}. In 2004, Alon, Balogh, Keevash and Sudakov~\cite{ABKS} settled this conjecture in full, proving that, for all $k\ge 3$ and sufficiently large $n$, the Tur\'an graph $T_{n,k-1}$ maximizes the number of $2$-edge-colorings and $3$-edge-colorings with no monochromatic $K_k$ among all graphs:
\begin{eqnarray}\label{eq-ER-orig-lower}
F(n,2,k)=2^{\ex(n,K_k)} \quad \mbox{ and } \quad  F(n,3,k)=3^{\ex(n,K_k)}. 
\end{eqnarray}
For $4$-edge-colorings, the only two known cases are when $k=3,4$: an asymptotic result was given in~\cite{ABKS} for $k=3,4$; the exact result was proved by Pikhurko and Yilma~\cite{Oleg}, who showed that $T_{n,4}$ and $T_{n,9}$ maximize the number of $4$-edge-colorings with no monochromatic $K_3$ and $K_4$ respectively, see~\cite{Oleg-K} for more recent development.

Since the Tur\'an graph is extremal in the Erd\Ho s-Rothschild problem for $r=2,3$, it is natural to consider its Ramsey-Tur\'an extension. Formally, given a function $f(n)$, we define $\RF(r,k,f(n))$ to be the maximum number of $r$-edge-colorings that an $n$-vertex graph with independence number at most $f(n)$ can have without a monochromatic copy of $K_k$. Similarly, the trivial lower bound on $\RF(r,k,o(n))$ is given by taking all edge-colorings on an extremal graph for Ramsey-Tur\'an problem:
\begin{eqnarray}\label{eq-ER-lower}
\RF(r,k,o(n))\ge r^{\RT(n,K_k,o(n))+o(n^2)}.
\end{eqnarray}
Considering~\eqref{eq-ER-orig-lower}, it is not inconceivable that the lower bound in~\eqref{eq-ER-lower} is optimal when $r$ is small. However, as shown in the following example, $\RF(r,k,o(n))$ exhibits rather different behavior than $F(n,r,k)$, even in the 2-edge-coloring case when $K_4$ is forbidden. Let $G$ be a graph obtained by putting a copy of $\Gamma$ in each partite set of $T_{n,2}$, where $\Gamma$ is a triangle-free graph with independence number $o(n)$.\footnote{The existence of $K_k$-free graph $\Gamma$ with $\al(\Gamma)=o(|\Gamma|)$ was proved by Erd\Ho s~\cite{Erdos-girth}.} Since $\Gamma$ is triangle-free, the neighborhood of every vertex is an independent set. Therefore, the independence number of the graph $\Gamma$ is at least its maximum degree, which implies that $\Gamma$ has maximum degree $o(n)$. Consider the following set of $2$-edge-colorings of $G$. Color the edges inside one partite set red, the edges inside the other partite set blue, and color all the remaining cross-edges either red or blue. It is not hard to see that none of these colorings contain monochromatic $K_4$'s,
hence, $\RF(2,4,o(n))\ge 2^{n^2/4}$, while $\RT(n,K_4,o(n))=(1/8+o(1))n^2$.

The above example already suggests an obstacle in determining $\RF(2,k,o(n))$, that is, the subgraphs induced by each color could simultaneously have linear-sized independent set. Nonetheless, our first result reveals the asymptotic behavior of $\RF(2,k,o(n))$ for every integer $k$. 

\begin{theorem}\label{thm-2coloringall}
$\RF(2,3,o(n))=2^{o(n^2)}$. For every integer $t\ge 1$ and $i\in\{1,2,3\}$, 
\begin{eqnarray*}
\RF(2,3t+i,o(n))= 2^{\RT(n,K_{4t+i},o(n))+o(n^2)}.
\end{eqnarray*}
\end{theorem}

The following well-known theorem determines the asymptotic value of $\RT(n,K_k,o(n))$, for every $k\ge 3$. For odd $k$, this was proved by Erd\H os and S\'os~\cite{ES}. For $k=4$, the upper bound is due to Szemer\'edi~\cite{Sz-k4}. In~\cite{BE}, Bollob\'as and Erd\H os showed that this upper bound is asymptotically sharp. These results were extended by Erd\H os, Hajnal, S\'os, and Szemer\'edi~\cite{EHSSz} to every even $k$.

\begin{theorem}\label{thm-RT}
	For every integer $k\ge 3$,
	\begin{eqnarray*}
		\RT(n,K_k,o(n))= (b_k+o(1))n^2,
	\end{eqnarray*}
	where
	\begin{eqnarray}\label{eq-bk}
	b_{k}=
	\left\{
	\begin{array}{ll}
	\frac{1}{2}.\frac{k-3}{k-1}&\mbox{\quad\quad if } k\mbox{ is odd}, \\
	
	\\
	
	\frac{1}{2}.\frac{3k-10}{3k-4}&\mbox{\quad\quad if } k\mbox{ is even}.
	\end{array}
	\right.
	\end{eqnarray}
\end{theorem} 

The definition of $b_{k}$ comes from optimizing the number of edges in a construction that we will describe in Section~\ref{sec-pre} (Construction~\ref{const-oddeven}). By Theorem~\ref{thm-RT}, to prove Theorem~\ref{thm-2coloringall}, it suffices to prove the following theorem.

\begin{theorem}\label{thm-2coloringreduction}
$\RF(2,3,o(n))=2^{o(n^2)}$. For every integer $t\ge 1$ and $i\in\{1,2,3\}$, 
\begin{eqnarray*}
	\RF(2,3t+i,o(n))= 2^{(b_{4t+i}+o(1))n^2}.
\end{eqnarray*}
\end{theorem}

Note that  since the value of the Ramsey-Tur\'an function is only known asymptotically, we will not try to determine the exact value of $\RF(2,k,o(n))$.
Our constructions for the lower bound in Theorem~\ref{thm-2coloringreduction} are based on the Bollob\'as-Erd\H os graph~\cite{BE}. 

\subsection{A generalized Ramsey-Tur\'an problem}
The generalized Tur\'an-type problem, i.e.~for given graphs $F$ and $H$, determine $\ex(n,F,H)$, the maximum number of copies of $F$ in an $n$-vertex $H$-free graph, has been studied for various choices of $F$ and $H$. Erd\H os~\cite{Erd} determined $\ex(n,K_s,K_t)$ for all $t>s\ge 3$, showing that among all $K_t$-free graphs, $T_{n,t-1}$ has the maximum number of $K_s$'s. See also Bollob\'as and Gy\H ori~\cite{BG} for $\ex(n,K_3,C_5)$, and more recently, Alon and Shikhelman~\cite{AS} for the cases when $(F,H)$ are $(K_3,C_5)$, $(K_m,K_{s,t})$, and when both $F$ and $H$ are  trees.

Our second result studies the general function $\RT(F,H,f(n))$, which is the maximum number of copies of $F$ in an $H$-free $n$-vertex graph $G$ with $\al(G)\le f(n)$. It is not hard to see that $\RT(K_s,K_{s+1},o(n))=o(n^s)$. We determine, in the following two theorems, $\RT(K_3,K_t, o(n))$ for every integer $t$.

\begin{theorem}\label{thm-k3}
Let $t\ge 6$ be an integer and $\ell=\left\lfloor\frac{t}{2}\right\rfloor$. Then as $n$ tends to infinity,
\begin{eqnarray*}
\RT(K_3,K_t, o(n))=a_{\ell}n^3(1+o(1)),
\end{eqnarray*}
where
\begin{eqnarray}\label{eq-al}
a_{\ell}=
\left\{
	\begin{array}{ll}
		\underset{0 \leq x \leq 1}\max {\ell-2\choose 3}\left(\frac{1-x}{\ell-2}\right)^3+
		 x{\ell-2\choose 2}\left(\frac{1-x}{\ell-2}\right)^2+
		 \frac{1}{2}\cdot\left(\frac{x}{2}\right)^2\left(1-x\right)
		 &\mbox{\quad\quad if } t=2\ell, \\
		\\
				\left(\frac{1}{\ell}\right)^3{\ell\choose 3}  &\mbox{\quad\quad if } t=2\ell+1.
	\end{array}
\right.
\end{eqnarray}
\end{theorem}
In fact, our proof shows that all the extremal graphs should have the structure as those in~Construction~\ref{const-oddeven}. The definition of $a_{\ell}$ comes from optimizing the number of $K_3$'s in these graphs.

For the general case $t>s\ge 3$, we present a construction in the concluding remark which we believe gives the right answer. Our next result verifies the first non-trivial case.
\begin{theorem}\label{thm-ks}
For every $s\ge 3$,	
$$\RT(K_s,K_{s+2},o(n))=\left(2^{-{s\choose 2}}+o(1)\right)\left(\frac{n}{s}\right)^s.$$
\end{theorem} 

\medskip

\noindent\textbf{Organization.} We first introduce some tools in Section~\ref{sec-pre}. Then in Section~\ref{sec-2coloring}, we prove Theorem~\ref{thm-2coloringreduction}, and in Section~\ref{sec-ks}, Theorems~\ref{thm-k3} and~\ref{thm-ks}. 

\medskip

\noindent\textbf{Notation.} Let  $G=(V,E)$ be an $n$-vertex graph and $e(G)=|E(G)|$. For every $v\in V$ and 
$U,U'\subseteq V$, denote by $d_{U}(v)$ the degree of $v$ in $U$. Also, let $N_{U}(v)$ be the set of 
vertices $u\in U$ such that $vu\in E(G)$. Denote by $G[U,U']$ the induced bipartite subgraph of $G$ on partite sets $U$ and $U'$. Let $k_s(G)$ be the number of $K_s$ in $G$. For every 
$A\subseteq V(G)$ and an \mbox{$r$-coloring} of $E(G)$ with colors $\{c_1,\ldots, c_r\}$, let $G_{c_i}[A]$ be 
the $c_i$-colored subgraph of $G$ induced by the vertex set $A$. We will write $G_{c_i}$ instead of 
$G_{c_i}[V(G)]$. We fix throughout the paper a function $\omega(n)$ of $n$ such that $\omega(n)\rightarrow\infty$ arbitrary slowly. If we claim that a result holds whenever $0 <  a \ll b \leq 1$, then this means that there is a non-decreasing
function $f : (0, 1] \rightarrow (0, 1]$ such that the result holds for all $0 < a, b\leq1 $ with $a \leq f(b)$.


\section{Preliminaries}\label{sec-pre}
We start with a formal definition for $\RT(n,H,o(n))$.
\begin{definition}
	For a graph $H$ and a function $f(n)$, let
	\begin{eqnarray*}
		\RT(n,H,o(f(n)))=n^2\cdot
		\lim\limits_{\ep\rightarrow 0}\lim\limits_{n\rightarrow\infty}\frac{\RT(n,H,\ep f(n))}{n^2}+o(n^2).
	\end{eqnarray*}
	
\end{definition}

Bollob\'as and Erd\Ho s~\cite{BE} constructed a family of $n$-vertex $K_4$-free graphs with independence number $o(n)$ and $(\frac{1}{8}+o(1))n^2$ edges. We follow the description in~\cite{SS} to present their construction. For a constant $\ep>0$, and sufficiently large integers $d$ and $n_0$, assume $n>n_0$ is even and $\mu=\ep/\sqrt d$. Next, partition the high-dimensional unit sphere $\bS^d$ into $n/2$ domains, $D_1,\ldots,D_{n/2}$, of equal measure with diameter\footnote{The diameter is the maximum distance between any two points in each domain.} less than $\mu/2$. For every $1\le i\le n/2$, choose two points $x_i, y_i\in D_i$. Let $X=\{x_1,\ldots,x_{n/2}\}$ and $Y=\{y_1,\ldots,y_{n/2}\}$. Let $\BE(X,Y)$ be the graph with vertex set $X\cup Y$ and edge set as follows. For every $x,x'\in X$ and $y,y'\in Y$,
\begin{enumerate}
\item let $xy\in E(\BE(X,Y))$ if their distance is less than $\sqrt 2-\mu$,
\item let $xx'\in E(\BE(X,Y))$ if their distance is more than $2-\mu$,
\item let $yy'\in E(\BE(X,Y))$ if their distance is more than $2-\mu$.
\end{enumerate} 

Note that the number of edges with both ends in $X$ or $Y$ is $o(n^2)$.  

Next, for every integer $k\ge 3$, we will describe a family of $n$-vertex $K_k$-free graphs with independence number $o(n)$. As we mentioned earlier, the constant $b_k$ defined in~\eqref{eq-bk} comes from maximizing the number of edges in the construction below. In other words, some of these graphs are extremal graphs for Theorem~\ref{thm-RT}, i.e.~they have $(b_k+o(1))n^2$ edges.

\begin{construction}\label{const-oddeven}
	Given $k\ge 3$, denote by $\cH(n,k)$ the family of $n$-vertex graphs $G$ obtained as 
	follows. Let $\ell=\floor{\frac{k}{2}}$, and $\Gamma_n$ be an $n$-vertex triangle-free graph with 
	$\al(\Gamma_n)=o(|\Gamma_n|)$. If $k$ is odd, start 
	with a complete balanced $\ell$-partite graph on vertex set $V_1\cup\ldots\cup 
	V_{\ell}$. Then put a copy of $\Gamma_{|V_i|}$ in each $V_i$. If $k$ is even, 
	partition the vertex set into $\ell$ parts $\{V_1,\ldots,V_\ell\}$, such that $|V_1|=|V_2|$. First, let $G[V_1\cup V_2]$ be a copy of the Bollob\'as-Erd\Ho s graph $\BE(V_1,V_2)$; then for every $i\in \{1,\ldots,\ell\}$ and $j\in\{3,\ldots,\ell\}\setminus\{i\}$, let $G[V_i,V_j]$ be a complete bipartite graph; next, for every $i\in\{3\ldots\ell\}$, put a copy of 
	$\Gamma_{|V_i|}$ in each $V_i$.
\end{construction}

\begin{rmk}
	Note that in Construction~\ref{const-oddeven}, for even $k$, $|V_1|=|V_2|$ and  $\{V_1,\ldots,V_\ell\}$ is not necessarily an equipartition.  
\end{rmk}
We will also need the following definitions of regular partitions and weighted cluster graph.
\begin{definition}
	Let $G$ be a graph and $A,B\subseteq V(G)$. Denote by  $d(A,B):=\frac{e(G[A,B])}{|A||B|}$ the {\it density} of the pair $(A,B)$. Given $\ep>0$, a pair $X,Y\subseteq V(G)$ is \emph{$\ep$-regular} if for every $A\subseteq X$ and $B\subseteq Y$ with $|A|\ge \ep |X|$ and $|B|\ge \ep |Y|$, $|d(A,B)-d(X,Y)|\le \ep$. A vertex partition of $G$, $V(G)=C_1\cup \ldots\cup C_m$ is $\ep$-regular if all but $\ep m^2$ pairs of $(C_i,C_j)$ are $\ep$-regular.
\end{definition}

\begin{definition}
For every $\ep>0$, positive integer $t$, and an $n$-vertex graph $G=(V,E)$, let $\cC=\{C_1, \ldots, C_m\}$ be an $\ep$-regular partition of $V(G)$ with $m\ge t$. Denote by $R$ the \emph{cluster graph} (with respect to $\ep$) with vertex set $\cC$, and $C_i$ and $C_j$ are adjacent if the pair $(C_i,C_j)$ is $\ep$-regular with density at least $10\ep$. We now define the \emph{weighted cluster graph}, $R=(\cC,w)$ (with respect to $\ep$), on the vertex set $\cC$ as follows. For an $\ep$-regular pair $(C_i,C_j)$, we will define:
\begin{eqnarray*}
w(C_i,C_j):=\left\{
	\begin{array}{ll}
		0\mbox{\ \ \ \  if }d(C_i,C_j)\le 10\ep \mbox{ or }(C_i,C_j)\mbox{ is an irregular pair},\\
		\frac{1}{2}\mbox{\ \ \ \  if }10\ep<d(C_i,C_j)\le 1/2+10\ep,  \\
		1\mbox{\ \ \ \  if }d(C_i,C_j)>1/2+10\ep.
	\end{array}
\right.
\end{eqnarray*}
\end{definition}

\begin{definition}\label{def-weight}
A \emph{weighted graph} $G$ is an ordered triple $(V,E,w)$ where $E:={V\choose 2}$, set of all unordered pairs of vertices, and $w:E\rightarrow 
\{0,1/2,1\}$. Define $G_{1/2}=(V,E_{1/2})$ where $E_{1/2}=\{e\in E: w(e)\ge 1/2\}$ and $G_{1}=(V,E_{1})$ 
where $E_{1}=\{e\in E: w(e)=1\}$. Denote by $e(G)=\sum_{e\in E(G)}w(e)$.

For two weighted graphs $G=(V,E,w)$ and $G'=(V,E,w')$, define $G\cap G'=(V,E,w'')$ where $w''(e)=\min\{w(e),w'(e)\}$. For $X\subseteq Y\subseteq V$, we call $(X,Y)$ a \emph{weighted $(|X|,|Y|)$-clique} or \emph{weighted complete subgraph} of size $\ell$ if ${X\choose 2}\subseteq E_1$ and ${Y\choose 2}\subseteq E_{1/2}$ and $|X|+|Y|=\ell$. Also, let the \emph{weighted clique number} of $G$ be the size of the largest weighted complete subgraph of $G$. 

For a triangle $T=e_1e_2e_3$, let $w(T)=\prod_{i=1}^{3} w(e_i)$. Also, let 
$T(G)=\sum_{T\in\cT} w(T)$, where $\cT$ is the set of all triangles in $G$. For $X\subseteq V(G)$, denote $T(X)=\sum_{T\in G[X]}w(T)$.
\end{definition}

We need the following two lemmas and theorem from~\cite{EHSSz}, the first one has been proved in the proof of Theorem~2 in~\cite{EHSSz}.

\begin{lemma}\label{lem-completefree2}
For every $\ep>0$, there exist $\de>0$ and $n_0$ such that for every $n$-vertex graph $G$ with $n\ge n_0$, if its weighted cluster graph $R(\cC,w)$ with respect to $\ep$ contains a weighted clique $(X,Y)$ of size $\ell$ such that $\al(G[X])=\de n$, then $G$ contains a copy of $K_\ell$.
\end{lemma}

\begin{lemma}\label{lem-weightedclique}
For every $\ep>0$ and integer $k\ge 3$ there exists $n_0$ such that for every $n$-vertex weighted graph $G=(V,E,w)$ with $n\ge n_0$, if $G$ does not contain a weighted complete subgraph of size $k$, then
\begin{eqnarray*}
e(G)\le (b_{k}+\ep)n^2,
\end{eqnarray*} 
where $b_{k}$ is defined in~\eqref{eq-bk}.
\end{lemma}


We will use the following multicolored version of the Szemer\'edi regularity lemma (for example, see~\cite{KS}). 

\begin{theorem}\label{thm-color}
For every $\ep>0$ and integer $r$, there exists an $M$ such that for every $n>M$ and every $r$-coloring of the edges of an $n$-vertex graph $G$ with colors $\{c_1,\ldots,c_r\}$, there exists a partition of $V(G)$ into sets $V_1,\ldots, V_m$ with $||V_i|-|V_j||\le 1$, for some $1/\ep<m<M$, which is $\ep$-regular with respect to $G_{c_i}$ for every $1\le i\le r$. 
\end{theorem}


\section{Proof of Theorem~\ref{thm-2coloringreduction}}\label{sec-2coloring}
To overcome the obstacle that all subgraphs induced by each color could have linear-sized independent sets, we need the following simple, but somewhat surprising observation.
\begin{lemma}\label{lem-partition}
For every $0<c<1$, $r\ge 2$, and $a\le a_r(c):=c^{3\cdot 2^{r-2}-1}$ the following holds. Let $G$ be an $n$-vertex graph with $\al(G)\le an$ and an $r$-edge-coloring $C:E(G)\rightarrow \{c_1,\ldots,c_r\}$. Then there exists a partition $V(G)=C_1\cup\ldots\cup C_r$ such that $\al(G_{c_i}[C_i])\le cn$ for every $1\le i\le r$.
\end{lemma}
\begin{proof}
We fix a $c>0$, and use induction on the number of colors $r$. For the base case when $r=2$ and $a\le c^2$, if $\al(G_{c_i})\le cn$, for some $i\in\{1,2\}$, then we can partition $V(G)$ into $C_i=V(G)$ and $C_j=\emptyset$, where $j=\{1,2\}\setminus\{i\}$, finishing the proof. Therefore, we may assume $\al(G_{c_1})> cn$ and $\al(G_{c_2})> cn$. Let $X_0=V(G)$, $Y_0=\emptyset$. We iterate the following operation for $i\ge 1$. At step $i$, if $\al(G_{{c_1}}[X_{i-1}])\le cn$ then we will stop. Otherwise, let $I$ be a maximum independent set in $G_{c_1}[X_{i-1}]$. Since $\al(G)\le an$, we have $\al(G_{c_2}[I])\le an$. We define $X_i:=X_{i-1}\setminus I$ and $Y_i:=Y_{i-1}\cup I$. Notice that $\al(G_{c_2}[Y_i])\le \al(G_{c_2}[Y_{i-1}])+an$. Suppose the iteration stops after $k$ steps, i.e.~$\al(G_{c_1}[X_k])\le cn$, then $k\le \frac{n}{cn}=1/c$, which implies that $\al(G_{c_2}[Y_k])\le k\cdot an\le cn$ as desired.

For the inductive step, let us assume that the lemma holds for $r-1$ colors, where $r\ge 3$. In particular, we assume that for every $a>0$, $n'$-vertex graph $H$ with $\al(H)\le an'$, and $(r-1)$-edge-coloring of $H$ with colors $c'_1,\ldots,c'_{r-1}$, there exists a partition of $V(H)=C'_1\cup\ldots\cup C'_{r-1}$ such that $\al(H_{c'_i}[C'_i])\le a^{1/(3\cdot 2^{r-3}-1)}n'$ for all $1\le i\le r-1$.

Now, we will prove the lemma for $r$ colors. Fix an arbitrary $r$-edge-coloring of $G$, 
we can assume that $\al(G_{c_1})> cn$, otherwise $C_1=V(G)$ and $C_2=\ldots=C_r=\emptyset$. Let $C_{1,0}=V(G)$ and $C_{i,0}=\emptyset$ for all $2\le i\le r$. 
We iterate the following operation. At step $k$, if $\al(G_{c_1}[C_{1,k-1}])\le cn$ then we will stop. 
Otherwise, let $I$ be a maximum independent set of $G_{c_1}[C_{1,k-1}]$ with $n'> cn$ vertices. Since $\al(G)\le a\cdot n$, for some constant $a\le a_r(c)$, we have 
$\al(G_{c_2}[I]\cup\ldots\cup G_{c_r}[I])\le an<an'/c$. We can apply the induction hypothesis to the graph $G[I]$. Therefore, there exists a partition of $I=I_2\cup\ldots\cup I_r$ such that for every $2\le i\le r$, we have
\begin{eqnarray}\label{eq-indnum}
\al(G_{c_i}[I_i])\le\left(\frac{a}{c}\right)^{\frac{1}{3\cdot 2^{r-3}-1}}n'\le c^{\frac{3\cdot 2^{r-2}-2}{3\cdot 2^{r-3}-1}}n'\le c^{\frac{3\cdot 2^{r-2}-2}{3\cdot 2^{r-3}-1}}n.
\end{eqnarray}
Then, we define $C_{1,k}:=C_{1,k-1}\setminus I$ and $C_{i, k}:=C_{i,k-1}\cup I_i$, for $2\le i\le r$. By~\eqref{eq-indnum}, 
\begin{eqnarray*}
\al(G_{c_i}[C_{i,k}])\le \al(G_{c_i}[C_{i,k-1}])+c^{\frac{3\cdot 2^{r-2}-2}{3\cdot 2^{r-3}-1}}n. 
\end{eqnarray*}
Let us assume that the iteration stops after $l$ steps, i.e.~$\al(G_{c_1}[C_{1,l}])\le cn$. Note that $l\le \frac{n}{cn}=1/c$, which implies that for every $2\le i\le r$,
\begin{eqnarray*}
\al(G_{c_i}[C_{i,l}])\le l\cdot c^{\frac{3\cdot 2^{r-2}-2}{3\cdot 2^{r-3}-1}}n\le\frac{1}{c}\cdot c^{\frac{3\cdot 2^{r-2}-2}{3\cdot 2^{r-3}-1}}n=c^{\frac{3\cdot 2^{r-2}-2}{3\cdot 2^{r-3}-1}-1}n=cn.
\end{eqnarray*}
\end{proof}

\begin{proof}[Proof of Theorem~\ref{thm-2coloringreduction}]
(Lower bound) Fix an arbitrary $t\ge 1$. For each $i=1,2,3$, by Theorem~\ref{thm-RT}, there is a $\floor{\frac{4t+i}{2}}$-partite graph $G_i$, with partite sets $\{V_1,\ldots,V_{\lfloor(4t+i)/2\rfloor}\}$, in $\cH(n,4t+i)$ from Construction~\ref{const-oddeven} such that $e(G_i)=\RT(n,4t+i,o(n))=(b_{4t+i}+o(1))n^2$. We take the following set of colorings. 

\begin{itemize}
	\item $i=1$: Set $V(G_1)=V_1\cup \ldots\cup V_{2t}$. Color $G_1[V_p]$, for $1\le p\le t$, red, color $G_1[V_q]$, for $t+1\le q\le 2t$, blue, and all cross-edges in $G_1[V_p,V_q]$, $1\le p<q\le 2t$, in either red or blue.
	
	\item $i=2$: Set $V(G_2)=V_1\cup \ldots\cup V_{2t+1}$. Color $G_2[V_p]$, for $1\le p\le t+1$, red, color $G_2[V_q]$, for $t+2\le q\le 2t+1$, blue, and all cross-edges in $G_2[V_p,V_q]$, $1\le p<q\le 2t+1$, in either red or blue.
	
	\item $i=3$: Set $V(G_3)=V_1\cup \ldots\cup V_{2t+1}$. Color $G_3[V_p]$, for $1\le p\le t$, red, color $G_3[V_q]$, for $t+1\le q\le 2t+1$, blue, and all cross-edges in $G_3[V_p,V_q]$, $1\le p<q\le 2t+1$, in either red or blue.
	
\end{itemize}

In all three cases, the total number of edges inside all $V_i$'s is $o(n^2)$. Therefore, the total number of cross-edges is $RT(n,4t+i,o(n))-o(n^2)$, which implies that we obtain $2^{\RT(n,4t+i,o(n))-o(n^2)}$ $2$-edge-colorings. Hence, we are left to show that all these colorings are monochromatic $K_{3t+i}$-free. For $i=1,3$, note that every blue (red~resp.) clique have at most one vertex from each $V_p$ ($V_q$~resp.), and at most two vertices from each $V_q$ ($V_p$~resp.). Hence, the size of the largest blue (red~resp.) clique is at most $t+2\cdot(t+\floor{\frac{i}{2}})<3t+i$ ($2t+t+\floor{\frac{i}{2}}<3t+i$~resp.). For the case when $i=2$, fix arbitrary $p,q$ such that $1\le p\le t+1$ and $t+2\le q\le 2t+1$. Note that to get a blue clique, we can have at most 1 vertex from each $V_p$ and 2 vertices from each $V_q$, hence, the largest blue clique has size at most $1\cdot(t+1)+2\cdot t=3t+1$. For a red clique, we can have at most 1 vertex from each $V_q$, at most a $K_3$ from $V_1\cup V_2$, and at most 2 vertices from each $V_p$, $p\neq 1,2$. Thus, the largest red clique is of size at most $3+2\cdot (t-1)+1\cdot t=3t+1$.

\medskip

(Upper bound) We will prove that for a given constant $\ep>0$, there exists $\ga>0$ such that, for  sufficiently large $n$, the following holds. Let $G$ be an $n$-vertex graph with $\al(G)\le \ga^2 n$, the number of $2$-edge-colorings of $G$ without a monochromatic $K_{3t+i}$ is at most $2^{(b_{4t+i}+\ep)n^2}$, for $t\ge 1$ and $i=1,2,3$. Also, the number of $2$-edge-colorings with no monochromatic $K_3$ is at most $2^{\ep n^2}$. Throughout the proof, constants are chosen from right to left according to the following hierarchy:
\begin{eqnarray}\label{eq-hierarchy}
0<\ga\ll\ep_2\ll\frac{1}{n_1}\ll\ep_1\ll\ep<1.
\end{eqnarray}

Let $n_0$ be the constant returned from Lemma~\ref{lem-weightedclique} with $\ep_1$ playing the role of $\ep$ and choose $n_1\ge n_0$. Let $\de$ be the constant returned from Lemma~\ref{lem-completefree2} with $\ep_2$ playing the role of $\ep$ and choose $\ga<\de$.

For any fixed $2$-edge-coloring of $G$, $\phi: E(G)\rightarrow\{\phi_1,\phi_2\}$, apply Lemma~\ref{lem-partition} with $r=2$, $c=\ga$, and $a_2(c)=\ga^2$. Let $\{A,B\}$ be the resulting partition such that 
\begin{eqnarray}\label{eq-part}
\al(G_{\phi_1}[A])\le \ga n, \quad\mbox{ and }\quad \al(G_{\phi_2}[B])\le \ga n.
\end{eqnarray}
We then apply Theorem~\ref{thm-color}, with $\ep_2$ playing the role of $\ep$, to $G$ with coloring $\phi$, and let $\cP=\{P_1,\ldots,P_m\}$ be the resulting partition of $V(G)$, where $m\ge 1/\ep_2$. Note that we may assume the regularity partition $\cP$ refines the $\{A,B\}$-partition. Let $R_{\phi_1}$ and $R_{\phi_2}$ be the $\phi_1$-colored and $\phi_2$-colored weighted cluster graphs respectively, both on vertex set $\{p_1,\ldots,p_m\}$, where the vertex $p_i$ represents the vertex set $P_i$, for all $i\in [m]$. We have
\begin{eqnarray}\label{eq-partitioning}
 \mbox{number of ways to fix an }\{A,B\}\mbox{-partition of }V(G) &\le&2^n,\nonumber\\
 \mbox{number of ways to fix a }\cP\mbox{-partition of }V(G) &\le&m^n,\nonumber\\
  \mbox{number of ways to fix }R_{\phi_1}\mbox{ and } R_{\phi_2} &\le&\left(2^{m\choose 2}\right)^4.
 \end{eqnarray}

Now, we will count the number of colorings with a fixed $\{A,B\}$-partition, $\cP$-partition and weighted cluster graphs $R_{\phi_1}$ and $R_{\phi_2}$. First, note that the number of edges of the graph $G$ with both ends in one of the $P_i$'s, between irregular or sparse pairs is at most 
\begin{eqnarray}\label{eq-removededges}
m\cdot\left(\frac{n}{m}\right)^2+\ep_2\cdot m^2\cdot\left(\frac{n}{m}\right)^2+10\cdot\ep_2\cdot m^2\cdot\left(\frac{n}{m}\right)^2\le \ep_2n^2+\ep_2n^2+10\cdot\ep_2n^2.
\end{eqnarray}
Hence, the number of ways to color these edges is at most $2^{12\ep_2n^2}$. From now on, we will only consider the rest of the edges of $G$, i.e.~the edges between pairs of clusters that are adjacent in $R_{\phi_1}\cup R_{\phi_2}$. Note that there is a unique way to color edges in $R_{\phi_1}\De R_{\phi_2}$. Thus the number of $2$-edge-colorings corresponding to the fixed $\{A,B\}$-partition, $\cP$-partition and weighted cluster graphs $R_{\phi_1}$ and $R_{\phi_2}$ is at most 
\begin{eqnarray}\label{eq-clusteredges}
2^{\left(\frac{n}{m}\right)^2e(R_{\phi_1}\cap R_{\phi_2})+12\ep_2n^2}.
\end{eqnarray}
To complete the proof, it remains to show that 
\begin{itemize}[noitemsep,topsep=0pt]
	\item[(i)]\label{itm-3} when $\phi$ is monochromatic $K_3$-free, $e(R_{\phi_1}\cap R_{\phi_2})=0$, and

	\item[(ii)]\label{itm-3t+i} when $\phi$ is monochromatic $K_{3t+i}$-free for $t\ge 1$ and $i=1,2,3$,
\end{itemize}
\begin{eqnarray}\label{eq-mainedges}
e(R_{\phi_1}\cap R_{\phi_2})\le  (b_{4t+i}+\ep_1)\cdot m^2,
\end{eqnarray} 
where $b_{4t+i}$ is defined in~\eqref{eq-bk}. Indeed, since the choice of $G$ is arbitrary, (i) together with~\eqref{eq-partitioning} and~\eqref{eq-clusteredges}, implies 
\begin{eqnarray*}
\RF(2,3,\ga^2n)&\le& 2^n\cdot m^n\cdot 2^{4\cdot{m\choose 2}}\cdot 2^{12\ep_2n^2}\le 2^{\ep n^2},
\end{eqnarray*}
and~(ii), together with~\eqref{eq-partitioning} and~\eqref{eq-clusteredges}, implies
\begin{eqnarray*}
\RF(2,3t+i,\ga^2n)&\le& 2^n\cdot m^n\cdot 2^{4\cdot{m\choose 2}}\cdot 2^{(b_{4t+i}+\ep_1) n^2+12\ep_2n^2}\le 2^{(b_{4t+i}+\ep) n^2}.
\end{eqnarray*}
To see (i), notice that if there is an edge, say $uv\in E(R_{\phi_1}\cap R_{\phi_2})$, then, without loss of generality, we may assume that $u\in A$. Therefore, by setting $X=\{u\}$ and $Y=\{u,v\}$, it follows from~\eqref{eq-part} that we have a $\phi_1$-colored weighted $(1,2)$-clique $(X,Y)$ with $\al(G_{\phi_1}[X])\le \ga n$, which by Lemma~\ref{lem-completefree2} yields a monochromatic $K_3$ in $\phi$, a contradiction.

For (ii), suppose that~\eqref{eq-mainedges} is not satisfied. Since $m>1/\ep_2>n_1$, we can apply Lemma~\ref{lem-weightedclique} to the graph $R_{\phi_1}\cap R_{\phi_2}$, with $\ep_1$ playing the role of $\ep$. Hence, the graph $R_{\phi_1}\cap R_{\phi_2}$ has a weighted complete subgraph $(X,Y)$ of size $4t+i$, and we shall find a monochromatic $K_{3t+i}$ in $G$ using $(X,Y)$, which is a contradiction. Let $x=|X|$, $y=|Y|$ and $X=\{p_1\cup\ldots\cup p_{x}\}$. Without loss of generality, we may assume that $p_1\cup\ldots\cup p_{\ceiling{x/2}}:=X'\subseteq A$, i.e.~$\al(G_{\phi_1}[P_1\cup\ldots\cup P_{\ceiling{x/2}}])\le\ga n$. We have thus found a weighted clique $(X',Y)$ in $G_{\phi_1}$ such that $\al(G_{\phi_1}[X'])\le\ga n$. Hence, Lemma~\ref{lem-completefree2} shows that $G_{\phi_1}$ contains a copy of $K_{{\ceiling{x/2}}+y}$. 
\begin{claim}\label{cl-x-small}
$\lfloor \frac{x}{2}\rfloor\le t$.
\end{claim}
\begin{proof}
Suppose that $\lfloor \frac{x}{2}\rfloor\ge t+1$. Recall from the definition of weighted clique that $X\subseteq Y$, i.e.~$x\le y$, and $x+y=4t+i$. Thus,
$$4t+3\ge 4t+i=x+y\ge 2x\ge 4\left\lfloor \frac{x}{2}\right\rfloor\ge 4(t+1),$$
a contradiction.
\end{proof}
Claim~\ref{cl-x-small} then implies that the monochromatic clique corresponding to $(X',Y)$ we found in $G_{\phi_1}$ is of order
\begin{eqnarray}\label{eq-cliquesize}
{\ceiling{\frac{x}{2}}}+y=x+y-\floor{\frac{x}{2}}\ge 4t+i-t=3t+i,
\end{eqnarray}
a contradiction.
\end{proof}


\section{Proof of Theorems~\ref{thm-k3} and~\ref{thm-ks}}\label{sec-ks}
\begin{proof}[Proof of Theorem~\ref{thm-ks}](Lower bound) Let $D_1,\ldots,D_{n/s}$ be a partition of the high-dimensional unit sphere of equal measure with small diameter as in the Bollob\'as-Erd\H os graph construction. Let $G$ be an $n$-vertex graph with a balanced vertex partition $V_1$, $\ldots$, $V_s$, where each $V_i$ consists of one point from each of the $n/s$ domains $D_1,\ldots,D_{n/s}$. For every pair of distinct integers $i,j\in [s]$, let $G[V_i\cup V_j]$ be a copy of $\BE(V_i,V_j)$. Note that each $G[V_i]$ is triangle-free and each $G[V_i\cup V_j]$ is $K_4$-free. We claim that $G$ is $K_{s+2}$-free. Indeed, let $F$ be a largest clique in $G$ and let $g_i=|V(F)\cap V_i|$. Since $G[V_i]$ is triangle-free, each $g_i\le 2$. If $|V(F)|\ge s+2$, then there exists at least two indices $p,q$ such that $g_p=g_q=2$, which contradicts to $G[V_p,V_q]$ being $K_4$-free. We will count the number of $K_s$ with exactly one vertex from each $V_i$. Fix a vertex $v_1\in V_1$, a uniformly at random chosen $v_2\in V_2$ is adjacent to $v_1$ if $v_2$ is in the cap (almost a hemisphere) centered at $v_1$ with measure $1/2-o(1)$, which happens with probability $1/2-o(1)$. Now we fix a clique on vertex set $\{v_1,\ldots, v_{\ell-1}\}$ with $\ell\ge 2$ and $v_i\in V_i$. The number of vertices in $V_{\ell}$ that are in $\bigcap_{i=1}^{\ell-1}N(v_i)$ is at least $2^{-(\ell-1)}n/s-o(n)$. Therefore, we have
\begin{eqnarray*}
k_s(G)\ge\prod_{i=1}^{s}\left[\left(2^{-(i-1)}-o(1)\right)\frac{n}{s}\right]=\left(2^{-{s\choose 2}}-o(1)\right)\left(\frac{n}{s}\right)^s.
\end{eqnarray*}

(Upper bound) We will prove that for a given $\ep>0$ and integer $s\ge 3$, there exists $\ga>0$ such that for any sufficiently large $n$ the following holds. Let $G$ be an $n$-vertex $K_{s+2}$-free graph with $\al(G)\le \ga n$, the number of edges of copies of $K_s$ in  $G$ is at most $(2^{-{s \choose 2}}+\ep)(n/s)^s$.  Throughout the proof, constants are chosen from right to left according to the following hierarchy,
\begin{eqnarray*}
0<\ga\ll\ep_1\ll\ep<1.
\end{eqnarray*}

Let $\ep_1$ play the role of $\ep$ in Lemma~\ref{lem-completefree2}, and choose $\ga$ such that it is smaller than the resulting $\de$. First, we apply Theorem~\ref{thm-color}, with $\ep_1$ playing the role of $\ep$, to the graph $G$ and let $\cP=\{P_1,\ldots,P_m\}$ be the resulting partition of $V(G)$, where $m\ge 1/\ep_1$, and let $R$ be the weighted cluster graph with respect to $\ep_1$. We call an edge in $R$ \emph{heavy} if it has weight $1$. We claim that the graph $R$ does not contain any weighted $(1,s+1)$- or $(2,s)$-clique. Otherwise, we apply Lemma~\ref{lem-completefree2} to the graph $R$ with $\ep_1$ playing the role of $\ep$. Since $\al(G[X])\le\al(G)\le \ga(n)\le \de_1(n)$, $G$ contains a copy of $K_{s+2}$, a contradiction. In other words, we have that $R$ is $K_{s+1}$-free and does not have a copy of $K_{s}$ with at least one heavy edge.

Now, we can count the total number of copies of $K_s$ in $G$. Note that similarly to~\eqref{eq-removededges}, the total number of edges inside all clusters, between irregular pairs, or sparse pairs is at most $12\cdot\ep_1n^2$. Therefore, the total number of copies of $K_s$ with at least one such edge is at most $12\cdot\ep_1n^2\cdot n^{s-2}$. Since $R$ is $K_{s+1}$-free, by the result of Erd\H os~\cite{Erd}, it has at most $(m/s)^s$ copies of $K_s$. Also, since $R$ does not have a copy of $K_s$ with a heavy edge, it implies that each $K_s$ in $R$ has weight at most $(1/2+10\cdot\ep_1)^{s\choose 2}\le 2^{-{s\choose 2}}+\ep/2$, where the last inequality holds because $\ep_1$ is sufficiently small with respect to $\ep$.  Hence, we have that the number of $K_s$ in $G$ is at most
\begin{eqnarray*}
	\left(2^{-{s\choose 2}}+\frac{\ep}{2}\right)\cdot\left(\frac{n}{m}\right)^s\cdot\left(\frac{m}{s}\right)^s+12\cdot\ep_1n^s\le
\left(2^{-{s\choose 2}}+\ep\right)\left(\frac{n}{s}\right)^s.
\end{eqnarray*}
\end{proof}

The following lemma is the main step for proving Theorem~\ref{thm-k3}.

\begin{lemma}\label{lem-triangle}
For every integer $t\ge 4$ and $n$-vertex weighted graph $G=(V,E,w)$ (as in Definition~\ref{def-weight}) with no weighted complete subgraph of size $t$, we have
\begin{eqnarray*}
T(G)\le a_{t}n^3,
\end{eqnarray*}
where $a_t$ is as in~\eqref{eq-al}.
\end{lemma}

\begin{proof}
Let $G=(V,E,w)$ be an $n$-vertex weighted graph that satisfies the hypothesis and is extremal, i.e.~has the maximum number of triangles. First, we will apply two rounds of the so-called symmetrization method to the graph 
$G$. For $v,v'\in V(G)$, denote $T_v(G)=\sum_{v\in T\in G}w(T)$, the number of weighted triangles containing $v$. Similarly, define $T_{vv'}(G)=\sum_{v,v'\in T\in G}w(T)$. Let $V(G)=\{v_1,\ldots,v_n\}$ such that $T_{v_1}(G)\ge\ldots\ge T_{v_n}(G)$. Define $S_1(i,j)$, for $i\neq j\in [n]$ to be the following operation: if $v_iv_j\notin G_{1/2}$ then we replace $v_j$ 
with a copy of $v_i$, i.e.~change $w(v_jv_k)$ to $w(v_i,v_k)$ for all $k\neq i,j$. If $i<j$, then the number of triangles changes by $T_{v_i}(G)-T_{v_j}(G)\ge 0$. Since $G$ is extremal, we have that $T_{v_i}(G)=T_{v_j}(G)$ for any $v_iv_j\notin G_{1/2}$. Consequently, the following process, denoted by $S_1$, is finite: apply $S_1(i,j)$ for every $1\le i<j\le n$ with $v_iv_j\notin G_{1/2}$. Note that $S_1$ will not increase the weighted clique number and keep the same number of triangles. After $S_1$, in the resulting graph 
$v_iv_j\notin G_{1/2}$ is an equivalence relation. Denote by $\cA=\{A_1,\ldots, A_m\}$ the equivalence 
classes of this relation, i.e.~two vertices $u$ and $v$ are in the same class if and only if $uv\notin 
G_{1/2}$. Therefore, for fixed $1\le i,j\le m$, all the edges between $A_i$ and $A_j$ have equal 
weights, which we denote by $w(A_iA_j)$, and for all vertices $x,x'\in A_i$ and $y,y'\in A_j$, we have
\begin{eqnarray*}
T_x(G)=T_{x'}(G)\quad\mbox{and}\quad T_{xy}(G)=T_{x'y'}(G).
\end{eqnarray*}
Therefore, we can define $T_{A_i}(G)=T_x(G)$ and $T_{A_iA_j}(G)=T_{xy}(G)$.
  Note that if $(X,Y)$ is one of the largest weighted complete 
subgraphs of $G$, then $|Y|=m$. 

We  summarize the structure of $G$ as follows: Let $H$ be a weighted complete graph on vertex set $\{a_1,\ldots,a_m\}$ with all its edges having weight either $1$ or $1/2$, and $w(a_ia_j)=w(A_iA_j)$.  The graph $G$ is a blow-up of $H$ where we replace each $a_i$ with a set of $|A_i|$ vertices, and inside each $A_i$ the weight of all edges is zero. 

Our next goal is to show that a second round of symmetrization  can be carried out in $G$, in other words, in $H$, $w(a_ia_j)=1/2$ is an equivalence relation. Without loss of generality we may assume $T_{A_1}(G)\ge\ldots\ge T_{A_m}(G)$. For every $1\le i,j\le m$, define $S_2(i,j)$ to be the following operation: Change $w(A_jA_k)$ to $w(A_iA_k)$ for all $k\neq i,j$, and denote $G_{A_i}$ the resulting graph. Define $G_{A_j}$ analogously as the graph obtained from applying $S_2(j,i)$ to $G$. The following claim states that when $w(A_iA_j)=1/2$, we can replace vertices in $A_i$ with copies of vertices in $A_j$, or the other way around, without decreasing the number of triangles.
\begin{claim}\label{cl-2-symm}
For every pair of integers $1\le i<j\le m$ with $w(A_iA_j)=1/2$,

(i) $T_{A_i}(G)=T_{A_j}(G)$;

(ii) $k_3(G_{A_i})= k_3(G_{A_j})= k_3(G)$.
\end{claim}
\begin{proof}
Define 
\begin{eqnarray*}
T^o_{A_i}(G)=T_{A_i}(G)-T_{A_iA_j}(G)\quad\mbox{ and }\quad G'=G\setminus \{A_i\cup A_j\}.
\end{eqnarray*}
Since $T_{A_i}(G)\ge T_{A_j}(G)$, we have
\begin{eqnarray}\label{eq-To}
T^o_{A_i}(G)+T_{A_iA_j}(G)\ge T^o_{A_j}(G)+T_{A_iA_j}(G)\quad \Leftrightarrow\quad T^o_{A_i}(G)\ge T^o_{A_j}(G).
\end{eqnarray}

\noindent For (i), it suffices to show 
\begin{eqnarray}\label{eq-o}
T^o_{A_i}(G)=T^o_{A_j}(G).
\end{eqnarray}
Note that
\begin{eqnarray}
k_3(G)&=&k_3(G')+|A_i|\cdot T^o_{A_i}(G)+|A_j|\cdot T^o_{A_j}(G)+|A_i|\cdot |A_j|\cdot T_{A_iA_j}(G),\label{eq-G}\\
k_3(G_{A_i})&=& k_3(G')+(|A_i|+|A_j|)\cdot T^o_{A_i}(G)+|A_i|\cdot |A_j|\cdot T_{A_iA_j}(G_{A_i}).\label{eq-GAi}
\end{eqnarray}
Then, since $G$ is extremal,
\begin{eqnarray*}
0\ge k_3(G_{A_i})-k_3(G)=|A_j|\cdot (T^o_{A_i}(G)-T^o_{A_j}(G))+|A_i|\cdot |A_j|\cdot (T_{A_iA_j}(G_{A_i})-T_{A_iA_j}(G)).
\end{eqnarray*}
Therefore, by~\eqref{eq-To}, we only need to show $T_{A_iA_j}(G_{A_i})\ge T_{A_iA_j}(G)$. Let 
\begin{eqnarray*}
V_{1,1/2}&=&\left\{A_\ell: w(A_iA_\ell)=1\mbox{ and }w(A_jA_\ell)=\frac{1}{2}\right\},\\
V_{1/2,1}&=&\left\{A_\ell: w(A_iA_\ell)=\frac{1}{2}\mbox{ and }w(A_jA_\ell)=1\right\},\\
V_{1/2,1/2}&=&\left\{A_\ell: w(A_iA_\ell)=\frac{1}{2}\mbox{ and }w(A_jA_\ell)=\frac{1}{2}\right\},\\
V_{1,1}&=&\left\{A_\ell: w(A_iA_\ell)=1\mbox{ and }w(A_jA_\ell)=1\right\}.
\end{eqnarray*}
Denote by $|V_{p,q}|=\sum_{A_{\ell}\in V_{p,q}}|A_{\ell}|$ for $p,q\in\{1/2,1\}$. We have
\begin{eqnarray*}
T_{A_iA_j}(G_{A_i})-T_{A_iA_j}(G)=\left(\frac{1}{2}-\frac{1}{4}\right)|V_{1,1/2}|-\left(\frac{1}{4}-\frac{1}{8}\right)|V_{1/2,1}|=\frac{1}{4}|V_{1,1/2}|-\frac{1}{8}|V_{1/2,1}|.
\end{eqnarray*}
Therefore, it suffices to show $2|V_{1,1/2}|\ge |V_{1/2,1}|$. For the sake of contradiction, assume 
\begin{eqnarray}\label{eq-cont}
|V_{1/2,1}|> 2|V_{1,1/2}|.
\end{eqnarray}
We will show that~\eqref{eq-cont} contradicts the extremality of $G$. Note that
\begin{eqnarray}\label{eq-GAj}
k_3(G_{A_j})&=& k_3(G')+(|A_i|+|A_j|)\cdot T^o_{A_j}(G)+|A_i|\cdot |A_j|\cdot T_{A_iA_j}(G_{A_j}).
\end{eqnarray}
By~\eqref{eq-G},~\eqref{eq-GAi},~\eqref{eq-GAj}, and the extremality of $G$ we have
\begin{eqnarray}
k_3(G_{A_i})\le k_3(G)\Leftrightarrow\left(\frac{1}{4}|V_{1,1/2}|-\frac{1}{8}|V_{1/2,1}|\right)|A_i|\cdot|A_j|+|A_j|\cdot T^o_{A_i}(G)\le |A_j|\cdot T^o_{A_j}(G),\label{eq-k3i}\\
k_3(G_{A_j})\le k_3(G)\Leftrightarrow\left(\frac{1}{4}|V_{1/2,1}|-\frac{1}{8}|V_{1,1/2}|\right)|A_i|\cdot|A_j|+|A_i|\cdot T^o_{A_j}(G)\le |A_i|\cdot T^o_{A_i}(G).\label{eq-k3j}
\end{eqnarray}
Then~\eqref{eq-k3i} and~\eqref{eq-k3j} imply
\begin{eqnarray*}
\left(\frac{1}{4}|V_{1,1/2}|-\frac{1}{8}|V_{1/2,1}|\right)|A_i|+T^o_{A_i}(G)\le T^o_{A_j}(G),\\
\left(\frac{1}{4}|V_{1/2,1}|-\frac{1}{8}|V_{1,1/2}|\right)|A_j|+ T^o_{A_j}(G)\le  T^o_{A_i}(G).\end{eqnarray*}
Therefore,
\begin{eqnarray*}
&&\left(\frac{1}{4}|V_{1/2,1}|-\frac{1}{8}|V_{1,1/2}|\right)|A_j|\le T^o_{A_i}(G)-T^o_{A_j}(G)\le 
\left(\frac{1}{8}|V_{1/2,1}|-\frac{1}{4}|V_{1,1/2}|\right)|A_i|\\
&\Rightarrow&\frac{1}{8}|V_{1/2,1}|(2|A_j|-|A_i|)\le \frac{1}{8}|V_{1,1/2}|(|A_j|-2|A_i|){\overset{\eqref{eq-cont}}{<}}\frac{1}{16}|V_{1/2,1}|(|A_j|-2|A_i|)\\
&\Rightarrow&4|A_j|-2|A_i|<|A_j|-2|A_i|\quad \Rightarrow\quad 4|A_j|<|A_j|,
\end{eqnarray*}
a contradiction.

For (ii), by the extremality of $G$, it suffices to show that $k_3(G_{A_i})+k_3(G_{A_j})\ge 2k_3(G)$. By~\eqref{eq-o},~\eqref{eq-G},~\eqref{eq-GAi} and~\eqref{eq-GAj}, we have
$$k_3(G_{A_i})+k_3(G_{A_j})-2k_3(G)=|A_i||A_j|\cdot ( T_{A_iA_j}(G_{A_i})+T_{A_iA_j}(G_{A_j})-2T_{A_iA_j}(G)).$$
It is left to show that $T_{A_iA_j}(G_{A_i})+T_{A_iA_j}(G_{A_j})-2T_{A_iA_j}(G)\ge 0$. Indeed,
\begin{eqnarray*}
&&T_{A_iA_j}(G_{A_i})+T_{A_iA_j}(G_{A_j})-2T_{A_iA_j}(G)\\
&&=\frac{1}{2}\sum_{\substack{1\le k\le m,\\ k\neq i,j}}w(A_iA_k)^2|A_k|+
\frac{1}{2}\sum_{\substack{1\le k\le m,\\ k\neq i,j}}w(A_jA_k)^2|A_k|-
2\cdot\frac{1}{2}\sum_{\substack{1\le k\le m,\\ k\neq i,j}}w(A_iA_k)w(A_jA_k)|A_k|\\
&&=\frac{1}{2}\sum_{\substack{1\le k\le m,\\ k\neq i,j}}\left(w(A_iA_k)-w(A_jA_k)\right)^2|A_k|\ge 0.
\end{eqnarray*}
\end{proof}

Denote by $S_2$ the following process: let $\sigma$ be the lexicographical ordering of ${[m]\choose 2}$ and apply $S_2(i,j)$, according to $\sigma$, for all pairs $(i,j)$ with $w(A_iA_j)=1/2$ . By Claim~\ref{cl-2-symm}, $S_2$ is finite and keeps the number of triangles.
\begin{claim}
The operation $S_2$ does not change the weighted clique number of $G$.
\end{claim}
\begin{proof}
Let $(X,Y)$ be one of the largest weighted complete subgraphs of $G$ of size $\ell$. Note that $|Y|$ is still $m$.
Also, since we only repeat this operation for vertices $x$ and $y$ with $w(xy)=1/2$, the operation is 
not changing $|X|$ either. Hence, after repeated applications of this operation, the weighted clique number of $G$ will not change. 
\end{proof}

After applying $S_2$, we have an equivalence relation on classes $A_1,\ldots, A_m$, which naturally extends to $V(G)$. To be precise,
denote by $\cB=\{B_1, \ldots, B_{m'}\}$ the equivalence classes of this relation, i.e.~two vertices $u$ and $v$ are in the same class if and only if $uv\notin G_{1}$. Then, the $\cA$-partition is a 
refinement of the $\cB$-partition. More importantly, the size of the largest weighted complete subgraph is $m+m'$. 

We will next show that we can perform some transformations (Claims~\ref{claim-twoA} and~\ref{claim-oneB}) to get a more structured graph (as those in Construction~\ref{const-oddeven}) without increasing the weighted clique number and decreasing the number of triangles.

\begin{claim}\label{claim-twoA}
Each $B_i$ contains at most two $A_j$'s.
\end{claim}
\begin{proof}
Let us assume that $B_1$ contains $k$ $A_j$'s, $A_1,\ldots, A_k$, where $k\ge 3$. Denote by  $U$ the vertex 
set of $B_1$ and write $u=|U|$. 
Note that the edges between two $B_i$'s always have weight $1$ and the edges inside an $A_i$ have 
weight $0$ and all the other edges have weight $1/2$. We will divide the proof into three cases depending 
on the value of $k$. In each case, we will modify $B_1$ by splitting it into multiple parts. This modification will only change the 
weight of the edges with both ends in $U$ and also the equivalence classes $\cA$ and $\cB$. Then we 
need to prove that the weighted clique number did not increase, and the number of triangles increases. For the latter, since the weight of the edges with at least one end in $V\setminus U$ remain the 
same, we only need to show that the number of triangles with two or three vertices in $U$ did not 
decrease. Therefore, it suffices to show that both $e(U)$ and $T(U)$ increase.

\textbf{Case 1:} Assume $k\ge 5$, which implies $u\ge 5$. We will split vertices in $U$ into three parts, $B_{11}$, $B_{12}$ and 
$B_{13}$, such that $|B_{11}|\le|B_{12}|\le|B_{13}|\le |B_{11}|+1$. Also, define $A_i=B_{1i}$ for all 
$1\le i\le 3$. For every $u\in U$ and $v\in V\setminus U$, we will not change $w(uv)$. For all 
vertices $u,u'\in U$ if they belong to the same $B_{1i}$, let $w(uu')=0$, otherwise let $w(uu')=1$. The 
equivalence classes $\cA$ and $\cB$ will change to $\{A_1,A_2,A_3,A_{k+1},\ldots,A_{m}\}$ and 
$\{B_{11}, B_{12}, B_{13}, B_2,\ldots,B_{m'}\}$. Since $k\ge 5$, the number of classes in the $\cA$ partition 
decreased by at least two and the number of classes in the $\cB$ partition increased by exactly $2$, hence, 
the weighted clique number of $G$ will not increase. Now, we only need to show that the number of 
triangles in the graph $G$ increases. 
\begin{eqnarray*}
&\mbox{before:}&e(U)\le{k\choose 2}\frac{u^2}{k^2}\cdot\frac{1}{2}< \frac{u^2}{4},\\
&\mbox{after:}&e(U)=
\left\{
	\begin{array}{ll}
		3\cdot\frac{u^2}{9}=\frac{u^2}{3}& \mbox{if }u\equiv 0\text{ (mod } 3), \\
		\frac{(u-1)^2}{9}+2\cdot\frac{(u-1)(u+2)}{9}=\frac{u^2-1}{3}& \mbox{if }u\equiv 1\text{ (mod } 3), \\
		\frac{(u+1)^2}{9}+2\cdot\frac{(u-2)(u+1)}{9}=\frac{u^2-1}{3}& \mbox{if }u\equiv 2\text{ (mod } 3).
	\end{array}
\right.
\end{eqnarray*}
Therefore $e(U)$ increases for $u\ge 2$. Now, for $T(U)$ we have
\begin{eqnarray*}
&\mbox{before:}&T(U)\le{k\choose 3}\frac{u^3}{k^3}\cdot\frac{1}{8}\le \frac{u^3}{48},\\
&\mbox{after:}&T(U)=
\left\{
	\begin{array}{ll}
		\frac{u^3}{27} & \mbox{if }u\equiv 0\text{ (mod } 3), \\
		\frac{(u-1)(u-1)(u+2)}{27}& \mbox{if }u\equiv 1\text{ (mod } 3), \\
		\frac{(u-2)(u+1)(u+1)}{27}& \mbox{if }u\equiv 2\text{ (mod } 3),
	\end{array}
\right.
\end{eqnarray*}
which means that $T(U)$ increases if $u\ge 3$.

\textbf{Case 2:} Assume $k=4$ which implies $u\ge 4$. Let us split vertices in $U$ into three parts $A_{1}$, $A_{2}$ and $A_{3}$, such that $|A_{1}|\le|A_{2}|\le|A_3|\le |A_1|+1$. Also let $B_{11}=A_1\cup A_2$ and $B_{12}=A_3$. For all vertices $u,u'\in U$ if they are in different $B_{1i}$'s then $w(uu')=1$. If they are both in $B_{11}$ but in different $A_i$'s then $w(uu')=1/2$, and $w(uu')=0$ if they are in the same $A_i$. The equivalence classes $\cA$ and $\cB$ will change to $\{A_1,A_2,A_3,A_{5},\ldots,A_{m}\}$ and $\{B_{11}, B_{12}, B_2,\ldots,B_{m'}\}$. Notice that the number of classes in the $\cA$ partition 
decreased by one and the number of classes in $\cB$ increased by one, hence, 
the weighted clique number of $G$ will not change. For $e(U)$:
\begin{eqnarray*}
&\mbox{before:}&e(U)\le{4 \choose 2}\frac{u^2}{16}\cdot \frac{1}{2}=\frac{3u^2}{16},\\
&\mbox{after:}&e(U)=
\left\{
	\begin{array}{ll}
		\frac{1}{2}\cdot\frac{u^2}{9}+2\cdot\frac{u^2}{9} & \mbox{if }u\equiv 0\text{ (mod } 3), \\
		\frac{1}{2}\cdot\frac{(u-1)(u-1)}{9}+2\cdot\frac{(u-1)(u+2)}{9}& \mbox{if }u\equiv 1\text{ (mod } 3), \\
		\frac{1}{2}\cdot\frac{(u-2)(u+1)}{9}+\frac{(u-2)(u+1)}{9}+\frac{(u+1)(u+1)}{9}& \mbox{if }u\equiv 2\text{ (mod } 3).
	\end{array}
\right.
\end{eqnarray*}
For $u\ge 2$, $e(U)$ increases. We also need to show that $T(U)$ increases:
\begin{eqnarray*}
&\mbox{before:}& T(U)\le 4\cdot\frac{u^3}{4^3}\cdot\frac{1}{8}=\frac{u^3}{4^3}\cdot\frac{1}{2},\\
&\mbox{after:}&T(U)=
\left\{
	\begin{array}{ll}
		\frac{1}{2}\cdot\frac{u^3}{27} & \mbox{if }u\equiv 0\text{ (mod } 3), \\
		\frac{1}{2}\cdot\frac{(u-1)(u-1)(u+2)}{27}& \mbox{if }u\equiv 1\text{ (mod } 3), \\
		\frac{1}{2}\cdot\frac{(u-2)(u+1)(u+1)}{27}& \mbox{if }u\equiv 2\text{ (mod } 3).
	\end{array}
\right.
\end{eqnarray*}
Therefore $T(U)$ will increase for $u\ge 3$.

\textbf{Case 3:} Assume $k=3$, which implies $u\ge 3$. First, suppose that $u\le 12n/13$. Split vertices in $U$ into two equal parts $B_{11}$ and $B_{12}$. Also define $A_1=B_{11}$ and $A_2=B_{12}$. For all vertices $u,u'\in U$ if they are in different $B_{1i}$'s then set $w(uu')=1$, and \mbox{$w(uu')=0$} otherwise. The equivalence classes $\cA$ and $\cB$ will change to $\{A_1,A_2,A_{4},\ldots,A_{m}\}$ and $\{B_{11}, B_{12}, B_2,\ldots,B_{m'}\}$. Notice that the number of classes in the $\cA$ partition 
decreased by one and the number of classes in the $\cB$ partition increased by one, hence, 
the weighted clique number of $G$ will not change. For the change on the number of triangles, we have
\begin{eqnarray*}
&\mbox{before:}& T(U)+e(U)(n-u)\le \frac{u^3}{3^3}\cdot\frac{1}{2^3}+\frac{1}{2}\cdot{3\choose 2}\frac{u^2}{9}(n-u),\\
&\mbox{after:}&T(U)+e(U)(n-u)=
\left\{
	\begin{array}{ll}
		0+\frac{u^2}{4}\cdot(n-u) & \mbox{if }u\text{ is even},\\
		0+\frac{(u-1)(u+1)}{4}\cdot(n-u)\ge\frac{u^2}{4.5}(n-u)& \mbox{if }u\text{ is odd}.
	\end{array}
\right.
\end{eqnarray*}
Since $u\le 12n/13$, we have
\begin{gather*}
\frac{u^3}{3^3}\cdot\frac{1}{2^3}+\frac{3}{2}\cdot(n-u)\cdot\frac{u^2}{3^2}\le (n-u)\cdot\frac{u^2}{4.5}\quad\Leftrightarrow\quad\frac{u^3}{6^3}\le (n-u)\cdot\frac{u^2}{18}\quad\Leftrightarrow\\ \frac{u}{12}\le (n-u)\quad\Leftrightarrow\quad u\le \frac{12}{13}n.
\end{gather*}
 
We may now assume that $u>12n/13$. Let $U'$ be the vertex set of 
$B_2$ and $u'=|B_2|$. Since $u\ge 12n/13$ and $u'\le n/13$, we may assume $B_2$ contains at most two $A_i$'s. Note that $u'\le u/12$. We split $U\cup U'$ into three classes of the same size, $B_{0}$, $B_{1}$ and $B_{2}$. Define $A_0=B_0$, $A_1=B_1$, and $A_2=B_2$. For two vertices $u,u'\in U\cup U'$, if they belong to the same $B_i$ then $w(uu')=0$, otherwise $w(uu')=1$. The equivalence classes $\cA$ and $\cB$ will change to $\{A_0,A_1,A_2, A_{5},\ldots,A_{m}\}$ and $\{B_{0}, B_{1}, B_2, B_3, \ldots,B_{m'}\}$. Notice that the number of classes in $\cA$ decreased by one and the number of classes in $\cB$ increased by one, which implies that the weighted clique number of $G$ will not change. We are left to 
show that this operation will increase $e(U\cup U')$ and $T(U\cup U')$:
\begin{eqnarray*}
&\mbox{before:}& e(U\cup U')\le\frac{u^2}{3^2}\cdot\frac{3}{2}+uu'+\frac{u'^2}{8}\le\frac{u^2}{6}+\frac{u^2}{12}+\frac{u'^2}{8}=\frac{3u^2}{12}+\frac{u'^2}{8},\\
&\mbox{after:}&e(U\cup U')=
\left\{
	\begin{array}{ll}
		3\cdot\frac{(u+u')^2}{9} & \mbox{if }u+u'\equiv 0\text{ (mod } 3), \\
		\frac{(u+u'-1)(u+u'-1)}{9}+2\cdot\frac{(u+u'-1)(u+u'+2)}{9}& \mbox{if }u+u'\equiv 1\text{ (mod } 3), \\
		2\cdot\frac{(u+u'-2)(u+u'+1)}{9}+\frac{(u+u'+1)^2}{9}& \mbox{if }u+u'\equiv 2\text{ (mod } 3).
	\end{array}
\right.
\end{eqnarray*}

\noindent Hence $e(U\cup U')$ is increasing for $u\ge 3$. We also have
\begin{eqnarray*}
&\mbox{before:}& T(U\cup U')\le\left(\frac{u}{3}\right)^3\cdot\frac{1}{8}+\frac{3}{2}\cdot\frac{u^2}{3^2}\cdot u'+u\cdot\frac{u'^2}{8}\le \frac{u^3}{6^3}+\frac{u^3}{6\cdot 12}+\frac{u^3}{8\cdot 12^2}\le\frac{u^3}{51.5},\\
&\mbox{after:}&T(U\cup U')= 
\left\{
	\begin{array}{ll}
		\frac{(u+u')^3}{27} & \mbox{if }u+u'\equiv 0\text{ (mod } 3), \\
		\frac{(u+u'-1)(u+u'-1)(u+u'+2)}{27}& \mbox{if }u+u'\equiv 1\text{ (mod } 3), \\
		\frac{(u+u'-2)(u+u'+1)(u+u'+1)}{27}& \mbox{if }u+u'\equiv 2\text{ (mod } 3).
	\end{array}
\right.
\end{eqnarray*}

Therefore $T(U\cup U')$ increases for $u+u'\ge 3$.
\end{proof}
\begin{claim}\label{claim-oneB}
There is at most one $B_i$ that contains two $A_j$'s.
\end{claim}
\begin{proof}
Now, we know that no $B_i$ contains three or more $A_i$'s. Let us assume that $B_1=A_1\cup A_2$ and $B_{2}=A_3\cup A_4$. Denote by $U$ the vertex set of $B_1\cup B_{2}$, and write $u=|U|$. Since each $A_i$ contains at least one vertex, we have that $u\ge 4$. We will split the vertices in $U$ into three equal pieces, $B_{11}$, $B_{12}$ and $B_{13}$, and redefine $A_1=B_{11}$, $A_2=B_{12}$, and $A_3=B_{13}$. For two vertices $u,u'\in U$ if they are in two different $B_{1i}$'s then $w(uu')=1$, otherwise $w(uu')=0$. This operation will change $\cA$ and $\cB$ to $\{A_1,A_2,A_3, A_{5},\ldots,A_{m}\}$ and $\{B_{11}, B_{12}, B_{13}, B_3, \ldots,B_{m'}\}$,  therefore the weighted clique number does not change. We only need to show that $e(U)$ and $T(U)$ increase.
\begin{eqnarray*}
&\mbox{before:}&e(U)\le\frac{u^2}{4^2}+\frac{u^2}{4}=\frac{5u^2}{16},\\
&\mbox{after:}&e(U)=
\left\{
	\begin{array}{ll}
		3\cdot\frac{u^2}{9} & \mbox{if }u\equiv 0\text{ (mod } 3), \\
		\frac{(u-1)^2}{9}+2\cdot\frac{(u-1)(u+2)}{9}& \mbox{if }u\equiv 1\text{ (mod } 3), \\
		\frac{(u+1)^2}{9}+2\cdot\frac{(u-2)(u+1)}{9}& \mbox{if }u\equiv 2\text{ (mod } 3),
	\end{array}
\right.\\
&\mbox{before:}& T(U)\le\frac{u^2}{16}\cdot\frac{u}{2}=\frac{u^3}{32},\\
&\mbox{after:}&T(U)=
\left\{
	\begin{array}{ll}
		\frac{u^3}{27} & \mbox{if }u\equiv 0\text{ (mod } 3), \\
		\frac{(u-1)^2(u+2)}{27}& \mbox{if }u\equiv 1\text{ (mod } 3), \\
		\frac{(u+1)^2(u-2)}{27}& \mbox{if }u\equiv 2\text{ (mod } 3).
	\end{array}
\right.
\end{eqnarray*} 

It can be easily checked that for $u\ge 4$, both $e(U)$ and $T(U)$ are not decreasing.
\end{proof}

Now, we will use the Claims~\ref{claim-twoA} and~\ref{claim-oneB} to complete the proof of  Lemma~\ref{lem-triangle}. Let us assume that the extremal graph has partitions $\cA=\{A_1,\ldots,A_m\}$ and $\cB=\{B_1,\ldots,B_{m'}\}$. Also, since $\cA$ is a refinement of $\cB$ and also by Claims~\ref{claim-twoA} and~\ref{claim-oneB}, we have $m'\le m\le m'+1$. When $t=2\ell+1$, the graph does not contain a weighted clique of size $2\ell+1$, which implies $m+m'\le 2\ell$. Therefore $m'=m=\ell$ will maximize the number of triangles. In particular, the extremal graph is an $\ell$-partite graph with partite sets $B_1\cup\ldots\cup B_{\ell}$, where $||B_i|-|B_j||\le 1$ for all $1\le i<j\le \ell$. Define $A_i=B_i$ for all $1\le i\le \ell$, and for two vertices $u$ and $v$ if they belong to two different $B_i$'s then $w(uv)=1$, otherwise $w(uv)=0$.

When $t=2\ell$, the graph does not contain a weighted clique of size $2\ell$ which implies $m+m'\le 2\ell-1$. Therefore, in the extremal example, $m'=\ell-1$ and $m=\ell$. Hence, the extremal example is an $(\ell-1)$-partite graph, with partite sets $B_1\cup\ldots\cup B_{\ell-1}$, and let $B_1=A_1\cup A_2$. Simple optimization shows that $|A_1|=|A_2|$, and, for all $2\le i\le \ell-1$, all the $B_i$'s have the same size, i.e.~$|B_1|=x$ and $|B_i|=(n-x)/(\ell-2)$ for all $2\le i\le\ell-1$. Fix two vertices $u$ and $v$, if they belong to two different $B_i$'s then set $w(uv)=1$. Otherwise, if they both 
belong to $B_1$ but to different $A_i$'s then set $w(uv)=1/2$, and $w(uv)=0$ in all other cases. Now, we only need to maximize the number of triangles with respect to $x$, which is exactly the optimization in~\eqref{eq-al}, showing that $T(G)\le a_t n^3$. This completes the proof of 
Lemma~\ref{lem-triangle}. 
 \end{proof}

\begin{proof}[Proof of Theorem~\ref{thm-k3}] 
For any given integer $t\ge 6$, let $\ell=\floor{\frac{t}{2}}$. The lower bound comes from $\cH(n,k)$ with $k=t$ in Construction~\ref{const-oddeven}. In this case, we solve an optimization problem to find the size of $V_i$'s that maximizes the number of triangles, which is how $a_\ell$ is defined in~\eqref{eq-al}.
    
For the upper bound, we will show that for any $\ep>0$, there exists $\de>0$ such that the following holds for sufficiently large $n$. Let $G$ be an $n$-vertex $K_t$-free graph with $\alpha(G)\le \de n$. Then $G$ has at most $(1+\ep)a_{\ell}n^3$ triangles. 

Choose constants $0\ll \de \ll\ep_1\ll\ep<1$. Let $R=R(\cC,w)$ be the weighted cluster graph obtained from applying Theorem~\ref{thm-color} to $G$ with $\ep_1$ playing the role of $\ep$. By Lemma~\ref{lem-completefree2}, we have that $R(\cC,w)$ does not contain a weighted clique of size $t$. Then the upper bound follows from Lemma~\ref{lem-triangle} and that 
$$k_3(G)\le T(R)\cdot\frac{n^3}{|R|^3}+\ep n^3\le (1+\ep)a_{\ell}n^3,$$
as desired, where the last term bounds the number of triangles in $G$ that do not correspond to a triangle in $R$.
\end{proof}



\section{Concluding remarks}
In this paper, we study the Ramsey-Tur\'an extensions of two special cases of classical problems. We determine $\RF(2,k,o(n))$, that is, the maximum number of $2$-edge-colorings an $n$-vertex graph with independence number $o(n)$ can have without a monochromatic $K_k$, and $\RT(K_3,K_t,o(n))$, the maximum number of triangles in an $n$-vertex $K_t$-free graph with $o(n)$ independence number.
\subsection{3-edge-colorings}
The Ramsey-Tur\'an extension of the Erd\Ho s-Rothschild problem for more than 2 colors remains widely open. It is known~\cite{ABKS} that $F(n,3,k)=3^{\ex(n,K_k)}$. It will be interesting to study for $3$-edge-colorings, $\RF(3,k,o(n))$. The following determines the case when forbidding monochromatic triangles. We give here only a sketch of a proof.
\begin{theorem}
$\RF(3,3,o(n))=2^{n^2/4+o(n^2)}$.
\end{theorem}
\begin{proof}[Sketch of a proof.] (Lower bound) Let $G\in \cH(n,5)$ with $|V_1|=|V_2|=n/2$. Consider the following 3-edge-colorings. Color edges in $G[V_i]$, $i=1,2$, red and color the cross-edges $G[V_1,V_2]$ either green or blue.
	
(Upper bound) Let $G$ be an extremal graph, $\phi: E(G)\rightarrow \{\phi_1,\phi_2,\phi_3\}$ be a $3$-edge-coloring with no monochromatic $K_3$, and $(A_1,A_2,A_3)$ be the partition obtained from Lemma~\ref{lem-partition} such that $\al(G_{\phi_i}[A_i])=o(n)$. Let $R^*$ be the multigraph by taking the union $\cup_{i=1}^3R_{\phi_i}$, where $R_{\phi_i}$ is the cluster graph in color $\phi_i$. 
Denote by $\mu_i$, $i=1,2,3$, the edge-density of the subgraph of $R^*$ induced by edges with multiplicity $i$. Note first that $\mu_3=0$, since otherwise a multiplicity-3 edge results in a weighted clique of size 3 in $\cap_i R_{\phi_i}$, contradicting to $\phi$ containing no monochromatic $K_3$. It suffices then to show $\mu_2\le 1/2$. Notice that no $\phi_i$-colored edge can have an endpoint in $A_i$, otherwise we have a $\phi_i$-colored triangle. This implies that 

(i) for every $i\neq j\in [3]$, all edges in $R^*[A_i,A_j]$ have multiplicity 1 with color $\phi_k$, $k\neq i,j$;

(ii) for $i\in [3]$, all edges in $R^*[A_i]$ have multiplicity at most 2, colored in $\{j,k\}=[3]\setminus\{i\}$.

By (i), we only need to consider edges in $\cup_i R^*[A_i]$. By (ii), inside $A_i$, there is no color $\phi_i$. This together with the observation that edges colored in $\{\phi_p,\phi_q\}$ for any $p\neq q\in [3]$ is triangle-free, we have $\mu_2\le 1/2$ as desired.
\end{proof}

\subsection{Generalized Ramsey-Tur\'an for larger cliques}
It seems plausible that for the general case $\RT(K_s,K_t,o(n))$, $t> s\ge 3$, some graph from the following construction has the maximum number of $K_s$.
\begin{construction}\label{const-Hst}
Given $3\le s<t\le 2s-1$, denote by $\cH(n,s,t)$ the family of $n$-vertex graphs $G$ on vertex set $V_1\cup\ldots\cup V_s$ obtained as follows. Let $H$ be an extremal $K_{t-s}$-free graph on vertex set $[s]$. Make $[V_i,V_j]$ complete bipartite if $ij\in E(H)$; otherwise, put a copy of $\BE(V_i,V_j)$ if $ij\not\in E(H)$. For every $i\in V(H)$ with $d_H(i)=s-1$, put a $|V_i|$-vertex triangle-free graph with $o(|V_i|)$ independence number in $V_i$.
\end{construction}

Note that all graphs $G$ in the above construction are $K_t$-free and have $o(n)$ independence number. Indeed, since $G[V_i]$ is triangle-free for all $i$, in order to have a copy of $K_t$, there should be at least $t-s$ classes, $V_{j_1},\ldots,V_{j_{t-s}}$, each containing 2 vertices that form a $K_{2(t-s)}$. This would imply for every $1\le p<q\le t-s$, $G[V_{j_p},V_{j_q}]$ contains a $K_4$, which contradicts to $H$ being $K_{t-s}$-free. It should also be noted that the sizes of $V_i$'s need to be optimized.

\begin{conj}
Given integers $t> s\ge 3$, one of the extremal graphs for $\RT(K_s,K_t,o(n))$ lies in $\cH(n,s,t)$ from Construction~\ref{const-Hst} when $t\le 2s-1$, and lies in $\cH(n,k)$ with $k=t$ from Construction~\ref{const-oddeven} when $t\ge 2s$.
\end{conj}

\subsection{Phase transition}

For a given graph $H$ and two functions $f(n)\le g(n)$, we say that the Ramsey-Tur\'an function for $H$  exhibits a {\em jump} or has a {\em phase transition} from $g(n)$ to $f(n)$ if
	\begin{eqnarray*}
		\limsup\limits_{n\rightarrow\infty}\frac{\RT(n,H,f(n))}{n^2}
		<\liminf\limits_{n\rightarrow\infty}\frac{\RT(n,H,g(n))}{n^2}.
	\end{eqnarray*} 

Let $g_r(n)=n2^{-\om(n)\log^{1-1/r}n}$. Balogh, Hu and Simonovits~\cite{BHS} showed that the Ramsey-Tur\'an function for the even clique $K_{2r}$ exhibits a jump from $o(n)$ to $g_r(n)$. A similar phenomenon happens in the more general setup.
\begin{theorem}\label{thm-jump}
	\begin{enumerate}
		\item\label{itm-jump56} $\RT(K_3,K_5,g_3(n))=o(n^3)$ and $\RT(K_3,K_6,g_3(n))=o(n^3)$.
		\item\label{itm-jumpodd} Odd cliques larger than 5 are stable: for every $\ell\ge 3$,
		$$\RT(K_3,K_{2\ell+1},g_{\ell+1}(n))=(1+o(1))\RT(K_3,K_{2\ell+1},o(n)).$$
		\item\label{itm-jumpeven} Even cliques always exhibit a jump: for every $\ell\ge 3$,
		$$\RT(K_3,K_{2\ell+2},g_{\ell+1}(n))=(1+o(1))\RT(K_3,K_{2\ell+1},o(n)).$$
	\end{enumerate}
\end{theorem}

We will need a lemma by Balogh-Hu-Simonovits (Claim 6.1 in~\cite{BHS}).

\begin{lemma}\label{lem-BHS}
	Let $G$ be an  $n$-vertex graph with $\al(G)=g_q(n)$, where $g_q(n)=n2^{-w(n)\log^{1-1/q}n}$ and $\omega(n)\rightarrow\infty$ arbitrary slowly. If there exists a $K_q$ in the cluster graph of $G$, then $K_{2q}\subseteq G$. 
\end{lemma}

\begin{proof}[Proof of Theorem~\ref{thm-jump}]
	For~(\ref{itm-jump56}), note that $\RT(K_3,K_5,g_3(n))\le\RT(K_3,K_6,g_3(n))$. Therefore, we only need to prove $\RT(K_3,K_6,g_3(n))=o(n^3)$. By Lemma~\ref{lem-BHS}, if $G$ is an $n$-vertex $K_6$-free graph with $\al(G)\le g_3(n)$ then the cluster graph of $G$ is $K_3$-free, which means that $k_3(G)=o(n^3)$.
	
	For~(\ref{itm-jumpodd}), note that $\RT(K_3,K_{2\ell+1},g_{\ell+1}(n))\le \RT(K_3,K_{2\ell+1}, o(n))$, hence, by Theorem~\ref{thm-k3}, it is sufficient to prove $\RT(K_3,K_{2\ell+1},g_{\ell+1}(n))\ge(1+o(1)){\ell\choose 3}\left(\frac{n}{\ell}\right)^3$. Construction~\ref{const-oddeven} shows that this inequality holds.
	
	For~(\ref{itm-jumpeven}), note that $\RT(K_3,K_{2\ell+1},g_{\ell+1}(n))\le \RT(K_3,K_{2\ell+2}, g_{\ell+1})$. Hence, using~(\ref{itm-jumpodd}), we only need to show that \mbox{$ \RT(K_3,K_{2\ell+2}, g_{\ell+1}(n))\le \RT(K_3,K_{2\ell+1},o(n))$}. By Lemma~\ref{lem-BHS},  if $G$ is an $n$-vertex $K_{2\ell+2}$-free graph with $\al(G)\le g_{\ell+1}(n)$ then the cluster graph of $G$ is $K_{\ell+1}$-free. Then, by the result of Erd\Ho s~\cite{Erd}, among all $K_{\ell+1}$-free graphs the $\ell$-partite Tur\'an graph has the maximum number of triangles. Hence, we have
	\begin{eqnarray*}
		\RT(K_3,K_{2\ell+2},g_{\ell+1}(n))\le \left(1+o(1)\right){\ell\choose 3}\left(\frac{n}{\ell}\right)^3=RT(K_3,K_{2\ell+1}, o(n)),
	\end{eqnarray*}
	where the last equality is by Theorem~\ref{thm-k3}.
\end{proof}


\section{Acknowledgement}
The authors would like to thank the anonymous referees for their careful reading and helpful comments. 

\end{document}